\documentclass[10pt,openany]{article}
\usepackage[a4paper,hmargin={4cm,4cm},vmargin={4cm,4cm}]{geometry}
\usepackage[english]{babel}

\usepackage{amssymb}
\usepackage{pstricks}
\usepackage{amsthm}
\usepackage{amsmath}
\usepackage{ifsym}
\usepackage{euscript}
\usepackage{fourier}
\newrgbcolor{myblue}{0 0.1 0.6}

\renewenvironment{itemize}
  {\begin{list}{$\triangleright$}{%
   \setlength{\parskip}{0mm}
   \setlength{\topsep}{.5\baselineskip}
   \setlength{\rightmargin}{0mm}
   \setlength{\listparindent}{0mm}
   \setlength{\itemindent}{0mm}
   \setlength{\labelwidth}{2.5ex}
   \setlength{\itemsep}{.5\baselineskip}
   \setlength{\parsep}{0mm}
   \setlength{\partopsep}{0mm}
   \setlength{\labelsep}{1ex}
   \setlength{\leftmargin}{\labelwidth+\labelsep}
   }}{%
   \end{list}\vspace*{-1.3mm}}

\def\E{\exists}
\def\A{\forall}
\def\models{\vDash}
\def\Th{\textrm{Th}}
\def\sm{\smallsetminus}
\def\Aut{\textrm{Aut}}
\def\tp{\textrm{tp}}
\def\IMP{\Rightarrow}

\def\IFF{\Leftrightarrow}
\def\imp{\rightarrow}

\def\iff{\leftrightarrow}
\def\niff{\leftrightarrow\llap{\raisebox{.2ex}{{\tiny$/$}}\hskip1.2ex}}
\def\neq{\parbox{1.5ex}{\hfil=\llap{\raisebox{-.1ex}{{\small$/\,$}}}}}

\def\P{\EuScript P}
\def\D{\EuScript D}

\def\X{\EuScript X}
\def\C{\EuScript C}
\def\U{\EuScript U}

\def\B{\EuScript B}
\def\K{\EuScript K}
\def\<{\langle}
\def\>{\rangle}
\def\0{\varnothing}
\def\theta{\vartheta}
\def\phi{\varphi}
\def\epsilon{\varepsilon}
\def\emph#1{{\myblue\boldmath\bfseries #1}}
\def\ssf#1{\textsf{\small #1}}

\newtheoremstyle{mio}
     {2\parskip}
     {0mm}
     {\sl}
     {}
     {\bfseries}
     {}
     {3mm}
     {\llap{\thmnumber{#2}\hskip2mm}\thmname{#1}\thmnote{\bfseries{} #3}}

\newtheoremstyle{liscio}
     {2\parskip}
     {0mm}
     {}
     {}
     {\bfseries}
     {}
     {3mm}
     {\llap{\thmnumber{#2}\hskip2mm}\thmname{#1}\thmnote{\bfseries{}#3}}
\newcounter{thm}

\theoremstyle{mio}
\newtheorem{theorem}[thm]{Theorem}
\newtheorem{corollary}[thm]{Corollary}
\newtheorem{proposition}[thm]{Proposition}
\newtheorem{lemma}[thm]{Lemma}
\newtheorem{fact}[thm]{Fact}
\newtheorem{definition}[thm]{Definition}

\theoremstyle{liscio}

\newtheorem{remark}[thm]{Remark}

\newtheorem{example}[thm]{Example}
\def\QED{\noindent\nolinebreak[4]\hfill\rlap{\ \ $\Box$}\medskip}
\renewenvironment{proof}[1][Proof]%
{\begin{trivlist}\item[\hskip\labelsep\textbf{#1}]}
{\QED\end{trivlist}}

\mathcode`\=="203D
\title{Elementary classes of finite \vc-dimension}
\author{Domenico Zambella}
\date{}

\begin{document}
\raggedbottom

\def\vc{\textsc{vc}}
\def\nip{\textsc{nip}}

\maketitle

\def\ceq#1#2#3{\parbox[t]{25ex}{$\displaystyle #1$}\parbox[t]{5ex}{$\displaystyle\hfil #2$}{$\displaystyle #3$}}

\vskip5ex
\begin{abstract}\noindent
Let $\U$ be a saturated model of inaccessible cardinality, and let $\D\subseteq\U$ be arbitrary. Let  $\<\U,\D\>$ denote the expansion of $\U$ with a new predicate for $\D$. Write $e(\D)$ for the collection of subsets $\C\subseteq\U$ such that $\<\U,\C\>\equiv\<\U,\D\>$. We prove that if  the \vc-dimension of $e(\D)$ is finite then $\D$ is externally definable.
\end{abstract}

\vskip5ex

Let $\U$ be a saturated model of signature $L$, and let $T$ denote its theory and $\kappa$ its cardinality. We require that $\kappa$ is uncountable, inaccessible, and larger than $|L|$. There is no blanket assumption on $T$. Throughout the following $z$ is a tuple of variables of finite length and the letters $\D$ and $\C$ denote arbitrary subsets of $\U^{|z|}$. As usua,l the letters $A,B$, \ldots \@ denote subsets of $\U$ of small cardinality.

Recall that $\D$ is \emph{externally definable\/} if $\D=\D_{p,\phi}$ for some global type $p\in S_x(\U)$ and some $\phi(x,z)\in L$, where 

\ceq{\hfill \D_{p,\phi}}{=}{\big\{a\in\U^{|z|}\ :\ \phi(x,a)\in p\big\}.}

Externally definable sets are ubiquitous in model theory, though they mainly appear in the form of global $\phi\mbox{-}$types (in fact, they are in one-to-one correspondence with these).  One important fact about externally definable sets has been proved by Shelah in~\cite{Sh}, generalizing a theorem of Baisalov and Poizat in~\cite{BP}. Assume $T$ is \nip\ and let $\U^{\text{Sh}}$ be the model obtained by expanding $\U$ with a new predicate for each externally definable set. Then $\Th(\U^{\text{Sh}})$ has quantifier elimination. A few proofs of this result are available, see~\cite{Pi} and~\cite{CS}. The proof in~\cite{CS},  by Chernikov and Simon, is relevant to us because it introduces the notion of \textit{honest definition\/} that will find an application here. The Shelah expansion of groups with \nip\ has been studied in \cite{CPS}.

To any set $\D$ we associate an expansion of $\U$ with a new $|z|\mbox{-}$ary predicate for $z\in\D$. We denote this expansion by $\<\U,\D\>$. We denote by \emph{$e(\D/A)$\/} the set $\big\{\C\ :\ \<\U,\C\>\equiv_A\<\U,\D\>\big\}$. We would like to know if there there are conditions on $e(\D/A)$ that characterize externally definable sets. Note that there are straightforward conditions that characterize definable sets. For example, $\D$ is definable if and only if $|e(\D/A)|=1$ for some $A$. 

By adapting some ideas in \cite{CS} (see also~\cite{Z}), in Corollary~\ref{corol_e(D/A)} we prove a sufficient condition for $\D$ to be externally definable, namely that it suffices that for some set of parameters $A$ the \vc-dimension of $e(\D/A)$ is finite. Though in general this is not a necessary condition, it characterizes external definability when $T$ is \textsc{nip} (see Corollary~\ref{corol_e(D/A)_NIP}). Finally, in the last two sections we use $e(\D)$ in an attempt to generalize the notion of non-dividing to sets.

\section{Notation}

Let $L$ be a first-order language. We consider formulas build inductively from the symbols in $L$ and the atomic formulas \emph{$t\in\X$}, where $\X$ is some second-order variable and $t$ is a tuple of terms. For the the time being, the logical connectives are first-order only (in the last section we will add second-order quantification). The set of all formulas is itself denoted by \emph{$L$\/} or, if parameters from $A$ are allowed, by \emph{$L(A)$}. When a second-order parameter is included (we never need more than one) we write \emph{$L(A;\D)$}. When $\phi(\X)\in L(A)$ and \mbox{$\D\subseteq\U^{|z|}$}, we write $\phi(\D)$ for the formula obtained by replacing $\X$ by $\D$ in $\phi(\X)$. The truth of $\phi(\D)$ is defined in the obvious way. Warning: the meaning of $\phi(\D)$  depends on whether the formula is presented as $\phi(\X)$ or as $\phi(x)$ (see the first paragraph of Section~\ref{approximations}).

We write \emph{$\C\equiv_A\D$\/} if the equivalence $\phi(\C)\iff\phi(\D)$ holds for all $\phi(\X)\in L(A)$. Then the class \emph{$e(\D/A)$\/} defined in the introduction coincides with the set $\big\{\C\subseteq\U^{|z|}\ :\ \C\equiv_A\D\big\}$.

We say that $M$ is \emph{$L(A;\C)\mbox{-}$saturated\/} if every finitely consistent type $p(x)\subseteq L(A;\C)$ is realized in $M$. If $\C$ is such that $\U$ is $L(A;\C)\mbox{-}$saturated for every $A$, we say that $\C$ is \emph{saturated}. In other words, $\C$ is saturated if the expansion $\<\U,\C\>$ is a saturated model.

\begin{proposition}\label{prop_e(D/A_o(D/A)}
For every $\D$ and every $A$ there is a saturated $\C$ such that $\C\equiv_A\D$. Moreover, if $\D$ and $\C$ are both saturated, then there is $f\in\Aut(\U/A)$ that takes $\D$ to $\C$.
\end{proposition}

\begin{proof} 
We prove that there is $\C\equiv_A\D$ such that expansion $\<\U,\C\>$ is saturated.
As $\kappa$ is a large inaccessible cardinal, there is a model $\<\U',\D'\>\equiv_A\<\U,\D\>$ that is saturated and of cardinality $\kappa$. Then there is an isomorphism $f:\U'\to\U$ that fixes $A$. Then $f[\D']=\C$ is the required saturated subset of $\U$. The second claim is clear by back-and-forth.
\end{proof}

Let $\Delta$ be a set of formulas and let $\<I,<_I\>$ be a linearly ordered set. We say that the sequence $\<a_i : i\in I\>$ is \emph{indiscernible in $\Delta$} if for every integer $k$ and two increasing tuples $i_1 <_I \cdots <_I i_k$ and $j_1 <_I \cdots <_I j_k$ and formula $\phi(x_1,\ldots,x_k)\in \Delta$, we have $\phi(a_{i_1},\dots,a_{i_k})\iff \phi(a_{j_1},\dots,a_{j_k})$. When $\Delta=L(A)$ we say that $\<a_i : i\in I\>$ is \emph{$A\mbox{-}$indiscernible}.

We denote by \emph{$o(\D/A)$\/} the set $\big\{f[\D]:f\in\Aut(\U/A)\big\}$, that is, the orbit of $\D$ under $\Aut(\U/A)$. If $o(\D/A)=\big\{\D\big\}$ we say that \emph{$\D$ is invariant over $A$}. A global type $p\in S_x(\U)$ is invariant over $A$ if for every $\phi(x,z)$ the set $\D_{p,\phi}$ is invariant over $A$. The main fact to keep in mind about global $A\mbox{-}$invariant types is that any sequence $\<a_i\,:\,i<\lambda\>$ such that $a_i\models p_{\restriction A,a_{\restriction i}}$ is an $A\mbox{-}$indiscernible sequence.

We assume that the reader is familiar with basic facts concerning \textsc{nip} theories as presented, e.g., in \cite[Chapter 2]{Sim}.

\section{Approximations}\label{approximations}

The set $\D\cap A^{|z|}$ is called the \emph{trace\/} of $\D$ over $A$. For every formula $\psi(z)\in L(\U)$ we define \emph{$\psi(A)$\/} $=$ $\psi(\U)\cap A^{|z|}$, that is, the trace over $A$ of the definable set $\psi(\U)=\big\{a\in\U^{|z|}:\psi(a)\big\}$. 

A set $\D$ is called \emph{externally definable\/} if there are a global type $p\in S_x(\U)$ and a formula $\phi(x,z)$ such that $\D=\{a\;:\;\phi(x,a)\in p\}$. Equivalently, a set $\D$ is externally definable if it is the trace over $\U$ of a set which is definable in some elementary extension of $\U$. This explains the terminology.

We prefer to deal with external definability in a different, though equivalent, way. 

\begin{definition}\label{def_epprox}
We say that $\D$ is \emph{approximable\/} by the formula $\phi(x,z)$ if for every finite $B$ there is a $b\in\U^{|x|}$ such that $\phi(b,B)=\D\cap B^{|z|}$. We may call the formula $\phi(x,z)$ the \emph{sort} of $\D$. If in addition we have that $\phi(b,\U)\subseteq\D$, we say that  $\D$ is \emph{approximable from below}. If  $\D\subseteq\phi(b,\U)$ we say that  $\D$ is \emph{approximable from above}.\QED
\end{definition} 

Approximability from below is an adaptation to our context of the notion of \textit{having an honest definition} in \cite{CS}.  The following proposition is clear by compactness.

\begin{proposition}\label{lem_approx=external}
For every $\D$ the following are equivalent:
\begin{itemize}
\item[1.] $\D$ is approximable;
\item[2.] $\D$ is externally definable.\QED 
\end{itemize}
\end{proposition}

\begin{example}
Let $T$ be the theory a dense linear orders without endpoints and let $\D\subseteq\U$ be an interval. Then $\D$ is approximable both from below and from above by the formula \mbox{$x_1<z<x_2$}.  Now let $T$ be the theory of the random graph. Then every $\D\subseteq\U$ is approximable and, when $\D$ has small cardinality, it is approximable from above but not from below.\QED
\end{example}

In Definition~\ref{def_epprox}, the sort $\phi(x,z)$ is fixed (otherwise any set would be approximable) but this requirement of uniformity may be dropped if the sets $B$ are allowed to be infinite.

\begin{proposition}\label{lem_approx_nonunif}
For every $\D$ the following are equivalent:
\begin{itemize}
\item[1.] $\D$ is approximable;
\item[2.] for every $B$ of cardinality $\le|T|$ there is $\psi(z)\in L(\U)$ such that $\psi(B)=\D\cap B^{|z|}$.
\end{itemize}
Similarly, the following are equivalent:
\begin{itemize}
\item[3.] $\D$ is approximable from below;
\item[4.]  for every $B\subseteq\D$ of cardinality $\le|T|$ there is $\psi(z)\in L(\U)$ such that $B^{|z|}\subseteq \psi(\U)\subseteq\D$.
\end{itemize}
\end{proposition}

\begin{proof}
To prove \ssf{2}$\IMP$\ssf{1}, for a contradiction assume \ssf{2} and $\neg$\ssf{1}. For each formula $\psi(x,z)\in L$ choose a finite set $B$ such that $\psi(b,B)\neq\D\cap B^{|z|}$ for every $b\in\U^{|x|}$. Let $C$ be the union of all these finite sets. Clearly $|C|\le|T|$.  By \ssf{2} there are a formula $\phi(x,z)$ and a tuple $c$ such that $\phi(c,C)=\D\cap C^{|z|}$, contradicting the definition of $C$.

The implication \ssf{1}$\IMP$\ssf{2} is obtained by compactness and the equivalence \ssf{3}$\IFF$\ssf{4} is proved similarly. 
\end{proof}

\begin{proposition}\label{prop_approx_el_eq}
If $\D$ is approximable of sort $\phi(x,z)$ then so is any $\C$ such that $\C\equiv\D$. The same holds for approximability from below and from above.
\end{proposition}

\begin{proof}
If the set $\D$ is approximable by $\phi(x,z)$ then for every $n$

\hfil$\displaystyle\A z_1,\dots,z_n\;\E x\ \bigwedge^n_{i=1}\big[\phi(x,z_i)\ \iff\ z_i\in\D\big]$. 

So the same holds for any $\C\equiv\D$. As for approximability from below, add the conjunct $\A z\,\big[\phi(x,z)\imp z\in\D\big]$ to the formula above, and similarly for approximability from above.
\end{proof}

\section{The Vapnik-Chervonenkis dimension}

We say that $u\subseteq\P(\U^{|z|})$ \emph{shatters\/} $B\subseteq\U^{|z|}$ if every $H\subseteq B$ is the trace over $B$ of some set $\D\in u$. The \emph{\vc-dimension\/} of $u$ is finite if there is some $n<\omega$ such that no set of size $n$ is shattered by $u$.

\begin{proposition}\label{lkjfhd}
The following are equivalent:
\begin{itemize}
\item[1.]  $e(\D/A)$ has finite \vc-dimension;
\item[2.]  $o(\C/A)$ has finite \vc-dimension for some (any) saturated $\C\equiv_A\D$.
\end{itemize}
\end{proposition}

\begin{proof}
\ssf{1}$\IMP$\ssf{2}. Clear because $o(\C/A)\subseteq e(\D/A)$.

\ssf{2}$\IMP$\ssf{1}. Let $\C$ be any saturated set such that $\C\equiv_A\D$. Let $B$ be a finite set that is shattered by $e(\D/A)$, namely such that every $H\subseteq B$ is the trace of some $\C_H\equiv_A\D$. By Proposition~\ref{prop_e(D/A_o(D/A)}, we can require that all these sets $\C_H$ are saturated. Then they all belong to $o(\C/A)$. It follows that if $e(\D/A)$ has infinite \vc-dimension so does $o(\C/A)$.
\end{proof}

We say that a sequence of sentences $\<\phi_i:i<\omega\>$ \emph{converges\/} if the truth value of $\phi_i$ is eventually constant.

\begin{lemma}\label{lem_nip->indiscernible_converge}
Assume that $o(\D/A)$ has finite \vc-dimension and let $\<a_i:i<\omega\>$ be any $A\mbox{-}$indi\-scernible sequence. Then $\<a_i\in\D:i<\omega\>$ converges.
\end{lemma}
\begin{proof}
Negate the conclusion and let $\<a_i:i\in\omega\>$ witness this. We show that $o(\D/A)$ shatters $\{a_i:i<n\}$ for arbitrary $n$, hence that $o(\D/A)$ has infinite \vc-dimension. Fix some $H\subseteq n$, and for every $h<n$ pick some $a_{i_h}$ such that $a_{i_h}\in\D$ if and only if $h\in H$. We also require that $i_0<\dots< i_{n-1}$. Let $f\in\Aut(\U/A)$ be such that $f:a_{i_0},\dots a_{i_{n-1}}\mapsto a_0,\dots a_{n-1}$. Then $a_h\in f[\D]$ if and only if $h\in H$.
\end{proof}

We abbreviate $\U\sm\C$ as \emph{$\neg\C$}. We write \emph{$\neg^i$} for $\neg\dots(i$ times$)\dots\neg$ and abbreviate  $\neg^i(\cdot\in\cdot)$ as \emph{$\notin^i$}. The following lemmas adapt some ideas from \cite[Section~1]{CS} to our context.

\begin{lemma}\label{lem_tipi_inv}
Assume that $\C$ is saturated and that $o(\C/A)$ has finite \vc-dimension. Let \mbox{$M\preceq\U$} be an $L(A;\C)\mbox{-}$saturated. Then every global $A\mbox{-}$invariant type $p(z)$ contains a formula $\psi(z)\in L(M)$ such that either $\psi(\U)\subseteq\C$ or $\psi(\U)\subseteq\neg\C$.
\end{lemma}

\begin{proof}
By lemma~\ref{lem_nip->indiscernible_converge} there is no  infinite sequence $\<b_i:i<\omega\>$ such that

\ceq{\ssf{1.}\hfill b_i}{\models}{p(z)|_{A,b_{\restriction i}}\ \ \wedge\ \ z\notin^i\C.}

Let $n$ be the maximal length of a sequence $\<b_i:i<n\>$ that satisfies \ssf{1}. Then 

\ceq{\hfill}{}{ p(z)|_{A,b_{\restriction n}}\,\imp\, z\notin^n\C.}

As $M$ is  $L(A;\C)\mbox{-}$saturated, we can assume further that $b_i\in M$. Also, by saturation we can replace $p(z)|_{A,b_{\restriction n}}$ with some formula $\psi(z)$. Then, if $n$ is even, $\psi(\U)\subseteq\C$, and if $n$ is odd $\psi(\U)\subseteq\neg\C$.
\end{proof}

Notice that $p(z)\in S(M)$ is finitely satisfied in $A\subseteq M$ if and only if it contains the type

\ceq{\#\hfill q(z)}{=}{\big\{\neg\phi(z)\in L(M)\ :\ \phi(A)=\0\big\}.}

With this notation in mind, we can state the following lemma.

\begin{lemma}\label{corol_approx_C}
Assume $\C$ is saturated and $o(\C/A)$ has finite \vc-dimension. Then there are two formulas $\psi_i(z)$, where $i<2$, such that $\psi_i(z)\imp z\notin^i\C$ and, if $q(z)$ is the type defined above, $q(z)\imp\psi_0(z)\vee\psi_1(z)$.
\end{lemma}

\begin{proof}
Let $M$ be an $L(A;\C)\mbox{-}$saturated model. By definition, for every $a\models q(z)$ the type $\tp(a/M)$ is finitely satisfiable in $A$ so it extends to a global invariant type. By Lemma~\ref{lem_tipi_inv}, $q(\U)$ is covered by formulas $\psi(z)\in L(M)$ such that either $[\psi(z)\imp z\in\C]$ or $[\psi(z)\imp z\notin\C]$. The conclusion follows by compactness.
\end{proof}

\begin{theorem}\label{thm_nip_qd}
Assume $\C$ is saturated and $o(\C/A)$ has finite \vc-dimension for some $A$. Then $\C$ is approximable from below and from above.
\end{theorem}
\begin{proof}
Let $B\subseteq\C$ be given. Enlarging $A$ if necessary, we can assume that $B\subseteq A$. Let $M$ and $q(z)\subseteq L(M)$ be as in $\#$ above. Trivially $A\subseteq q(\U)$, hence $B\subseteq\psi_0(\U)\subseteq\C$. The set $B$ has arbitrary (small) cardinality. Then by Lemma~\ref{lem_approx_nonunif}, $\C$ is approximable from below.

As for approximation from above, observe that this is equivalent to $\neg\C$ being approximable from below. As $\neg\C$ is also saturated and $o(\neg\C/A)$ has finite \vc-dimension, approximability from above follows.
\end{proof}

\begin{corollary}\label{corol_e(D/A)}
Assume $e(\D/A)$ has finite \vc-dimension for some $A$. Then $\D$ is approximable from below and from above.\QED
\end{corollary}
\begin{proof}
Let $\C\equiv_A\D$ be saturated. As $o(\C/A)$ also has finite \vc-dimension, from Theorem~\ref{thm_nip_qd} it follows that $\C$ is approximable from below and from above. Then by Proposition~\ref{prop_approx_el_eq} the same conclusion holds for $\D$.
\end{proof}

Recall that a formula $\phi(x,z)\in L$ is \textsc{nip} if $\big\{\phi(a,\U)\, :\, a\in\U^{|x|}\big\}$ has finite \vc-dimension. If this is the case, $\big\{\D_{p,\phi}\,:\,p\in S_x(\U)\big\}$, that is, the set of externally definable sets of sort $\phi(x,z)$, also has finite \vc-dimension. Now, observe that if $\D$ is any externally definable set and $\C\equiv\D$ then $\C$ is also externally definable and has the same sort as $\D$. Hence, if $\phi(x,z)$ is \textsc{nip}, $e(\D)\subseteq\big\{\D_{p,\phi}\,:\,p\in S_x(\U)\big\}$ has finite \vc-dimension. 

The theory $T$ is \textsc{nip} if in $\U$ every formula is \textsc{nip}. Hence we obtain the following characterization of externally definable sets in a \textsc{nip} theory:

\begin{corollary}\label{corol_e(D/A)_NIP}
Il $T$ is \textsc{nip} then the following are equivalent:
\begin{itemize}
\item[1.] $\D$ is approximable from below (in particular, externally definable);
\item[2.] $e(\D)$ has finite \vc-dimension.\QED
\end{itemize}
\end{corollary}

We conclude by mentioning the following corollary, which is a version of Proposition~1.7 of \ssf{[CS]} stated with different terminology. Note that it is not necessary to require that $T$ is \textsc{nip}.

\begin{corollary}\label{corol_CS}
If $\D$ is approximable by a \textsc{nip}\ formula, then $\D$ is approximable from below.
\end{corollary}
\begin{proof}
If $\D$ is approximable of sort $\phi(x,z)$, by Proposition~\ref{prop_approx_el_eq}, so are all sets in $e(\D)$.  If  $\phi(x,z)$ is \textsc{nip}, then  $e(\D)$ has finite \vc-dimension and Corollary~\ref{corol_e(D/A)} applies.
\end{proof}

Observe that, given a formula $\phi(x,z)$ that approximates $\D$, the proof of Corollary~\ref{corol_CS} does not give explicitely the formula $\psi(x,z)$ that approximates $\D$ from below.

\section{Lascar invariance}
\def\equivL{\stackrel{\smash{\scalebox{.5}{\rm L}}}{\equiv}}

\def\Autf{{\rm Autf}}
The content of the second part of the paper is only loosely connected to the previous sections. We introduce the notion of a \textit{pseudo-invariant set\/} which is connected to non-dividing but it is sensible for arbitrary subsets of $\U$. We assume that the reader is familiar with basic facts concerning Lascar strong types and dividing (see e.g., \cite{Sim}, \cite{Cas}, \cite{TZ}) though in this section we will recall everything we need.

If $o(\D/A)=\big\{\D\big\}$ we say that $\D$ is \emph{invariant over $A$}. We say that $\D$ is \emph{invariant\/} tout court if it is invariant over some $A$. We say that $\D$ is \emph{Lascar invariant over $A$\/} if it is invariant over every model $M\supseteq A$.

\begin{proposition}\label{prop_numero_quasi_invarianti}
There are at most $2^{2^{|L(A)|}}$ sets $\D$ that are Lascar invariant over $A$.
\end{proposition}

\begin{proof}
Let $N$ be a model containing $A$ of cardinality $\le|L(A)|$. Every Lascar invariant set over $A$ is invariant over $N$. The proposition follows as $|N|\le|L(A)|$, and there are at most $2^{2^{|N|}}$ sets invariant over $N$.
\end{proof}

\begin{proposition}\label{prop_Lascar_indiscernibles}
For every $\D$ and every $A\subseteq M$ the following are equivalent:
\begin{itemize}
\item[1.] $\D$ is Lascar invariant over $A$;
\item[2.] every set in $o(\D/A)$ is $M\mbox{-}$invariant;
\item[3.] $o(\D/A)$ has cardinality $<\kappa$;
\item[4.] every endless $A\mbox{-}$indiscernible sequence is indiscernible in $L(A;\D)$;
\item[5.] $c_0\in\D\iff c_1\in\D$ for every $A\mbox{-}$indiscernible sequence $c=\<c_i:i<\omega\>$.
\end{itemize}
\end{proposition}

\begin{proof}
The implication \ssf{1}$\IMP$\ssf{2} is clear because all sets in $o(\D/A)$ are Lascar invariant over $A$. To prove \ssf{2}$\IMP$\ssf{3} it suffices to note that there are fewer than $\kappa$ sets that are invariant over $M$. 

We now prove \ssf{3}$\IMP$\ssf{4}. Assume $\neg$\ssf{4}. Then we can find an  $A\mbox{-}$indiscernible sequence $\<c_i:i<\kappa\>$ and a formula $\phi(x)\in L(A;\D)$ such that $\phi(c_0)\niff\phi(c_1)$. Define

\ceq{\hfill E(x,y)}{\IFF}{ \psi(x)\iff\psi(y)}\ \ \ for every $\C\in o(\D/A)$ and every $\psi(x)\in L(A;\C)$.

Then $E(x,y)$ is an $A\mbox{-}$invariant equivalence relation. As $\neg E(c_0,c_1)$, indiscernibility over $A$ implies that $\neg E(c_i,c_j)$ for every $i<j<\kappa$. Then $E(x,y)$ has $\kappa$ equivalence classes. As $\kappa$ is inaccessible, this implies $\neg$\ssf{3}.

The implication \ssf{4}$\IMP$\ssf{5} is trivial. We prove \ssf{5}$\IMP$\ssf{1}. Suppose $a\equiv_Mb$ for some $M\supseteq A$. Let $p(z)$ be a global coheir of $\tp(a/M)=\tp(b/M)$. Let $c=\<c_i:i<\omega\>$ be a Morley sequence of $p(z)$ over $M,a,b$. Then both $a,c$ and $b,c$ are $A\mbox{-}$indiscernible sequences. So from \ssf{5} we obtain $a\in\D\iff c_0\in\D\iff b\in\D$ and, as $M$ is arbitrary, \ssf{1} follows. 
\end{proof}

As the number of $M\mbox{-}$invariant sets is at most $2^{2^{|M|}}$, we obtain the following corollary.

\begin{corollary} For every $\D$ the following are equivalent:
\begin{itemize}
\item[1.] $o(\D/A)$ has cardinality $<\kappa$;
\item[2.] $o(\D/A)$ has cardinality $\le 2^{2^{|L(A)|}}$.\QED
\end{itemize} 
\end{corollary}

\section{Dividing}
Though Definition~\ref{def_loc_cover} below does not make any assumptions on $\B$ and $u\subseteq\P\big(\U^{|z|}\big)$, it yields a workable notion only when $\B$ is invariant and $u$ is closed in a sense that we will explain. Moreover, for the proof of Lemma~\ref{Lem_dividing_indiscernibles} we need $\kappa$ to be a Ramsey cardinal, so this will a blanket assumption throughout this section. 

\begin{definition}\label{def_loc_cover} Let $u\subseteq\P\big(\U^{|z|}\big)$ and let $\B\subseteq\U^{|z|}$. We say that $u$ \emph{locally covers\/} $\B$ if for every $\K\subseteq\B$ of cardinality $\kappa$ and every integer $k$ there is a $\D\in u$ such that $k\le|\K\cap\D|$.\QED
\end{definition}

The subsets of $\P(\U^{|z|})$ that are definable by formulas $\phi(\X)\in L(A)$ form a base of clopen sets for a topology. The proposition below implies that this topology is compact.

\begin{proposition}
Let $p(\X)\subseteq L(A)$ be finitely consistent, that is, for every $\phi(\X)$ conjunction of formulas in $p(\X)$ there is a $\D\subseteq\U^{|z|}$ such that $\phi(\D)$. Then there is a set $\C$ such that $p(\C)$.
\end{proposition}

\begin{proof}
The proposition follows from the fact that every saturated model is resplendent, see~\cite[Th\'eor\`eme 9.17]{Poi}. But the reader may prefer to prove it directly by adapting the argument used in the proof of Proposition~\ref{prop_e(D/A_o(D/A)}.
\end{proof}

Notice that the topology introduced above is not T$_0$ because there are $\C\neq\D$ such that \mbox{$\C\equiv\D$}. However, it is immediate that taking the Kolmogorov quotient (i.e.\@ quotienting by the equivalence relation $\equiv$) gives a Hausdorff topology. Then there is no real need to distinguish between \textit{compactness\/} and  \textit{quasi-compactness}. 

We will say that the set $u\subseteq\P\big(\U^{|z|}\big)$ is \emph{closed\/} if it is closed in the topology introduced above. In other words, $u$ is closed if  $u=\big\{\D\, :\, p(\D)\big\}$ for some $p(\X)\subseteq L$.

\begin{remark} We may read Definition~\ref{def_loc_cover} as a generalization of non-dividing. Let us recall the definition of dividing. We say that the formula $\phi(x,b)$ divides over $A$ if there there is an infinite set $\K\subseteq o(b/A)$ such that $\{\phi(x,c)\,:\,c\in\K\}$ is $k\mbox{-}$inconsistent for some $k$. By compactness, there is no loss of generality if we require $|\K|=\kappa$. Let $u\subseteq\P(\U^{|z|})$ contain the externally definable sets of sort $\phi(x,z)$. Then the requirement that $\{\phi(x,c)\,:\,c\in\K\}$ is $k\mbox{-}$inconsistent can be rephrased as $|\K\cap\D|<k$ for every $\D\in u$. So we may conclude that the following are equivalent:
\begin{itemize}
\item[1.] the formula $\phi(x,b)$ does not divide over $A$;
\item[2.] $u$ locally covers $o(b/A)$.
\end{itemize}
Incidentally, note that $o(b/A)$ is $A\mbox{-}$invariant and that $u$ is a closed set.\QED
\end{remark}

We now need to use second-order quantifiers. The set of formulas containing second-order quantifiers is denoted by  \emph{$L^2\!$}, or \emph{$L^2\!(A;\D)$\/} when parameters occur. Second-order quantifiers are interpreted to range over $\P(\U^{|z|})$. The following fact is immediate but noteworthy.

\begin{fact}
Every formula $\phi(x)\in L^2\!(A)$ is $A\mbox{-}$invariant and consequently any $A\mbox{-}$indiscernible sequence is indiscernible in $L^2\!(A)$.\QED
\end{fact}

\begin{lemma}\label{Lem_dividing_indiscernibles}
Let $u\subseteq\P\big(\U^{|z|}\big)$ be a closed set and let $\B\subseteq\U^{|z|}$ be an $A\mbox{-}$invariant set. Then the following are equivalent:
\begin{itemize}
\item[1.] $u$ locally covers $\B$;
\item[2.] every $A\mbox{-}$indiscernible sequence $\<a_i:i<\omega\>\subseteq\B$ is contained in some $\D\in u$.
\end{itemize}
\end{lemma}
\begin{proof}

\ssf{1}$\IMP$\ssf{2}. Let $p(\X)\in L$ be such that $u=\big\{\D\,:\,p(\D)\big\}$. Assume $\neg$\ssf{2} and fix an  $A\mbox{-}$indiscernible sequence $\<a_i:i<\omega\>\subseteq\B$ such that $p(\X)\ \cup\ \{a_i\in\X\,:\,i<\omega\}$ is inconsistent. By compactness there are some $i_1,\dots, i_k$ and some $\phi(\X)\in p$ such that 

\hfil$\displaystyle\A\X\ \Big[\phi(\X)\imp\neg\bigwedge^k_{n=1}a_{i_n}\in\X\Big]$.

Extend $\<a_i:i<\omega\>$ to an $A\mbox{-}$indiscernible sequence $\<a_i\,:\,i<\kappa\>$. By indiscernibility, every $\D\in u$ contains fewer than $k$ elements of $\{a_i:i<\kappa\}\subseteq\B$. Hence $\neg$\ssf{1}.

\ssf{2}$\IMP$\ssf{1}. Assume $\neg$\ssf{1} and fix $\K\subseteq\B$ of cardinality $\kappa$ and an integer $k$ such that $|\K\cap\D|<k$ for every $\D\in u$. As $\kappa$ is a Ramsey cardinal, there is an $A\mbox{-}$indiscernible $\<a_i:i<\kappa\>\subseteq\K$. Then $\<a_i:i<\kappa\>$ may not be contained in any $\D\in u$, hence $\neg$\ssf{2}.
\end{proof}

We say that $\D$ is \emph{pseudo-invariant\/} over $A$ if $e(\D)$ locally covers $o(b/A)$ for every $b\in\D$.

\begin{proposition}\label{prop_Lascar_invariant->nondivide}
If $\D$ is Lascar invariant over $A$, then for every $\phi(w)\in L(A;\D)$ the set $\phi(\U)$ is pseudo-invariant over $A$.
\end{proposition}

\begin{proof}
Fix $\phi(w)\in L(A;\D)$ and let  $b\in\phi(\U)$. Let $\<a_i:i<\omega\>\subseteq o(b/A)$ be an indiscernible sequence and fix some $f\in\Aut(\U/A)$ such that $fa_0=b$. Then  $\<fa_i:i<\omega\>$ is indiscernible in $L(A;\D)$ by Proposition~\ref{prop_Lascar_indiscernibles}. Then $\<fa_i:i<\omega\>\subseteq\phi(\U)$. Hence $\<a_i:i<\omega\>\subseteq f^{-1}[\phi(\U)]$. Clearly, $f^{-1}[\phi(\U)]\in e(\phi(\U))$, so the proposition follows from Lemma~\ref{Lem_dividing_indiscernibles}.
\end{proof}

\begin{proposition}\label{prop_Lascar_inv=pseudo-inv}
Let $e(\D)$ have finite \vc-dimension. Then the following are equivalent:\nobreak
\begin{itemize}
\item[1.] $\D$ is Lascar invariant over $A$;
\item[2.] $\phi(\U)$ is pseudo-invariant over $A$ for every $\phi(w)\in L(A;\D)$;
\item[3.] $\D\times\!\neg\D$ is pseudo-invariant over $A$.
\end{itemize}
\end{proposition}
\begin{proof} \ssf{1}$\IMP$\ssf{2} holds for any $\D$ by Proposition~\ref{prop_Lascar_invariant->nondivide} and \ssf{2}$\IMP$\ssf{3} is obvious. 

\ssf{3}$\IMP$\ssf{1}. Assume $\neg$\ssf{1}. By Proposition~\ref{prop_Lascar_indiscernibles}, there is an $A\mbox{-}$indiscernible sequence $\<a_i:i<\omega\>$ such that $a_0\in\D\niff a_1\in\D$, say $a_0\in\D$ and $a_1\notin\D$. Assume \ssf{2} for a contradiction. Then by Lemma~\ref{Lem_dividing_indiscernibles} there is $\C\equiv\D$ such that $\<a_{2i}a_{2i+1}\,:\, i<\omega\>\subseteq\C\times\!\neg\C$.  By Lemma~\ref{lem_nip->indiscernible_converge}, $e(\C)=e(\D)$ has infinite \vc-dimension contradicting the assumptions. 
\end{proof}

The hypothesis of finite \vc-dimension is necessary. Assume $T$ is the theory of dense linear orders without endpoints. Let $\D$ be a discretely ordered subset of $\U$ of cardinality $\kappa$. Then $\D$ is not invariant and $e(\D)$ has infinite \vc-dimension. One can verify that $\D\times\!\neg\D$ is pseudo-invariant over $\0$ directly from the definition.

It is well known that under the hypothesis that $T$ is \nip, Lascar invariance of global types is equivalent to non-dividing (equivalently, non-forking), see \cite[Proposition 5.21]{Sim}. Then, when $T$ is \nip, a global type $p(x)$ does not divide over $A$ if and only if $\D_{p,\phi}\times\!\neg\D_{p,\phi}$ is pseudo-invariant over $A$ for every $\phi(x,z)$.

However, pseudo-invariance is too strong a requirement to coincide with non-dividing in general. A counter-example may be found even when $T$ is simple. Let $T$ be the theory of the random graph and let $\D$ be a complete subgraph of $\U$. Let $p(x)$ be the unique global type that contains 

\hfil$\big\{r(x,a)\,:\, a\in\D\big\}\ \cup\ \big\{\neg r(x,a)\,:\,a\notin\D\big\}\ \cup\ \big\{x\neq a\,:\,a\in\U\big\}$. 

Then $p(x)$ does not fork over the empty set. On the other hand, $\D$ is not pseudo-invariant: let $\<a_i:i<\omega\>$ be an indiscernible sequence such that $a_0\in\D \wedge\neg r(a_0,a_1)$. As every $\C\equiv\D$ is a complete graph, no such $\C$ may contain $\<a_i:i<\omega\>$.

\vfil
\hfill\begin{minipage}{28ex}
Domenico Zambella\\
Dipartimento di Matematica\\
Universit\`a di Torino\\
via Calrlo Alberto 10\\
10123 Torino\\
\texttt{domenico.zambella@unito.it}
\end{minipage}

\end{document}